\documentclass[12pt]{article}

\usepackage[colorlinks]{hyperref}
\usepackage{graphicx,subfigure,amsmath,amssymb,amsfonts,bm,epsfig,epsf,url}
\usepackage{times}
\usepackage{bbm}
\usepackage{tikz}
\usepackage{booktabs}
\usepackage[lined,boxed,ruled,norelsize,algo2e]{algorithm2e}
\usepackage{epstopdf}
\DeclareGraphicsExtensions{.eps,.png}

\usepackage[small,bf]{caption}
\usepackage[top=1in,bottom=1in,left=1in,right=1in]{geometry}

\usetikzlibrary{arrows,intersections}
\usetikzlibrary{shapes.misc}

\tikzset{cross/.style={cross out, draw=black, minimum size=2*(#1-\pgflinewidth), inner sep=0pt, outer sep=0pt},
cross/.default={1pt}}

\renewcommand{\phi}{\varphi}

\newcommand{\mB}{\mathrm{B}}

\newcommand{\R}{\mathbb{R}}

\def\ds1{\mathds{1}}
\renewcommand{\epsilon}{\varepsilon}

\newcommand{\argmin}{\mathop{\mathrm{argmin}}}

\newlength{\minipagewidth}
\newcommand{\beq}{\begin{equation}}
\newcommand{\eeq}{\end{equation}}

\newcommand{\beqa}{\begin{eqnarray}}
\newcommand{\eeqa}{\end{eqnarray}}

\newcommand{\beqan}{\begin{eqnarray*}}
\newcommand{\eeqan}{\end{eqnarray*}}

\def\ba#1\ea{\begin{align*}#1\end{align*}} 
\def\banum#1\eanum{\begin{align}#1\end{align}} 

\newcommand{\LS}{\mathtt{line\_search}}

\newtheorem{theorem}{Theorem}
\newtheorem{lemma}{Lemma}

\newcommand{\BlackBox}{\rule{1.5ex}{1.5ex}}  
\newenvironment{proof}{\par\noindent{\bf Proof\ }}{\hfill\BlackBox\\[2mm]}

\begin{document}

\title{A geometric alternative to Nesterov's accelerated gradient descent}

\author{S\'ebastien Bubeck\\
Microsoft Research\\
sebubeck@microsoft.com
\and
Yin Tat Lee\footnote{Most of this work were done while the author was at Microsoft Research, Redmond. The author was supported by NSF awards 0843915 and 1111109.}\\
MIT\\
yintat@mit.edu
\and
Mohit Singh\\
Microsoft Research\\
mohits@microsoft.com
}
\date{\today}

\maketitle

\begin{abstract}
We propose a new method for unconstrained optimization of a smooth and strongly convex function, which attains the optimal rate of convergence of Nesterov's accelerated gradient descent. The new algorithm has a simple geometric interpretation, loosely inspired by the ellipsoid method. We provide some numerical evidence that the new method can be superior to Nesterov's accelerated gradient descent. 
\end{abstract}

\section{Introduction}
Let $f : \R^n \rightarrow \R$ be a $\beta$-smooth and $\alpha$-strongly convex function. Thus, for any $x,y\in \R^n$, we have
\begin{equation*}\label{eq:smooth-strong}
 f(x) + \nabla f(x)^{\top} (y-x) + \frac{\alpha}{2} |y-x|^2 \leq  f(y) \leq f(x) + \nabla f(x)^{\top} (y-x) + \frac{\beta}{2} |y-x|^2.
\end{equation*}
 Let $\kappa = \beta/\alpha$ be its condition number.  It is a one line calculation to verify that a step of gradient descent on $f$ will decrease (multiplicatively) the squared distance to the optimum by $1-1/\kappa$. In this paper we propose a new method, which can be viewed as some combination of gradient descent and the ellipsoid method, for which the squared distance to the optimum decreases at a rate of $(1-1/\sqrt{\kappa})$
 (and each iteration requires one gradient evaluation and two line-searches). This matches the optimal rate of convergence among the class of first-order methods, \cite{Nes83, Nes04}.

\subsection{Related works}
Nesterov's acceleration (i.e., replacing $\kappa$ by $\sqrt{\kappa}$ in the convergence rate) has proven to be of fundamental importance both in theory and in practice, see e.g. \cite{Bub14} for references. However the intuition behind Nesterov's accelerated gradient descent is notoriously difficult to grasp, and this has led to a recent surge of interest in new interpretations of this algorithm, as well as the reasons behind the possibility of acceleration for smooth problems, see \cite{AO14, LRP14, SBC14, FB15}.

In this paper we propose a new method with a clear intuition and which achieves acceleration. Since the function is strongly convex, gradient at any point gives a ball, say $A$, containing the optimum solution. Using the fact that the function is smooth, one can get an improved bound on the radius of this ball. The algorithm also maintains a ball $B$ containing the optimal solution obtained via the information from previous iterations. A simple calculation then shows that the smallest ball enclosing the intersection of $A$ and $B$ already has a radius shrinking at the rate of $1-\frac{1}{\kappa}$.  To achieve the accelerated rate, we make the observation that the gradient information in this iteration can also be used to shrink the ball $B$ and therefore, the radius of the enclosing ball containing the intersection of $A$ and $B$ shrinks at a faster rate. 
We detail this intuition in Section \ref{sec:intuition}. The new optimal method is described and analyzed in Section \ref{sec:optimal}. We conclude with some experiments in Section \ref{sec:exp}. 
 

\subsection{Preliminaries}
We write $|\cdot|$ for the Euclidean norm in $\R^n$, and $\mB(x,r^2) := \{y \in \R^n : |y-x|^2 \leq r^2 \}$ (note that the second argument is the radius squared). We define the map $\LS : \R^n\times \R^n \rightarrow \R^n$ by
$$\LS(x,y) = \argmin_{t \in \R} f(x + t (y-x)) ,$$
and we denote
$$x^+ = x - \frac{1}{\beta} \nabla f(x), \ \text{and} \ x^{++} = x - \frac{1}{\alpha} \nabla f(x) . $$
Recall that by strong convexity one has
$$\forall y \in \R^n, f(y) \geq f(x) + \nabla f(x)^{\top} (y-x) + \frac{\alpha}{2} |y-x|^2,$$
which implies in particular:
$$x^* \in \mB\left(x^{++}, \frac{|\nabla f(x)|^2}{\alpha^2} - \frac{2}{\alpha} (f(x) - f(x^*)) \right) .$$
Furthermore recall that by smoothness one has $f(x^+) \leq f(x) - \frac{1}{2 \beta} |\nabla f(x)|^2$ which allows to \emph{shrink} the above ball by a factor of $1-\frac{1}{\kappa}$ and obtain the following:
\begin{equation} \label{eq:ball2}
x^* \in \mB\left(x^{++}, \frac{|\nabla f(x)|^2}{\alpha^2} \left(1 - \frac{1}{\kappa}\right) - \frac{2}{\alpha} (f(x^+) - f(x^*)) \right) 
\end{equation}

\section{Intuition} \label{sec:intuition}

\begin{figure}[ht]
\centering

\begin{minipage}[t]{.45\textwidth}

\begin{tikzpicture}[scale=0.7, every node/.style={transform shape}]

\draw  (0,0) ellipse (2 and 2);
\draw[dashed]  (4,0) ellipse (4 and 4);
\draw  (4,0) ellipse (3.85 and 3.85);
\draw [|-|] (4,0) -- (4,4) node[pos=0.5, right] {$|g|$};
\draw [|-|] (4,0) -- (7.85,0) node[pos=0.5, above] {$\sqrt{1-\epsilon}\ |g|$};

\begin{scope}
  \clip (0,0) ellipse (2 and 2);
  \fill[lightgray] (4,0) ellipse (3.85 and 3.85);
\end{scope}

\draw (0,0) node[cross=5pt] {};

\draw [|-|] (0,0) -- (-2,0) node[pos=0.4, above] {$1$};
\draw [|-|] (0.64719,0) -- (2.53958,0) node[pos=0.35, above] {$\sqrt{1-\epsilon}$};
\draw[thick]  (0.64719,0) ellipse (1.8924 and 1.8924);

\end{tikzpicture}

\caption{One ball shrinks.}
\label{fig:one_ball}

\end{minipage}
\hfill
\noindent
\begin{minipage}[t]{.45\textwidth}

\begin{tikzpicture}[scale=0.7, every node/.style={transform shape}]

\draw[dashed]  (0,0) ellipse (2 and 2);
\draw  (0,0) ellipse (1.68 and 1.68);
\draw[dashed]  (4,0) ellipse (4 and 4);
\draw  (4,0) ellipse (3.85 and 3.85);
\draw [|-|] (4,0) -- (7.85,0) node[pos=0.5, above] {$\sqrt{1-\epsilon}\ |g|$};

\begin{scope}
  \clip (0,0) ellipse (1.68 and 1.68);
  \fill[lightgray] (4,0) ellipse (3.85 and 3.85);
\end{scope}

\draw (0,0) node[cross=4pt] {};

\draw [|-|] (0,-3) -- (-1.68,-3) node[pos=0.5, above] {$\sqrt{1-\epsilon |g|^2}$};
\draw [|-|] (0.5,0) -- (2.1044,0) node[pos=0.3, above] {\scriptsize $\sqrt{1-\sqrt{\epsilon}}$};

\draw[thick]  (0.5,0) ellipse (1.6044 and 1.6044);

\end{tikzpicture}

\caption{Two balls shrink.}
\label{fig:two_ball}

\end{minipage}

\caption*{The left diagram shows the intersection shrinks at the same rate if only one of the ball shrinks; the right diagram shows the intersection shrinks much faster if two balls shrinks at the same absolute amount.}

\end{figure}

In Section \ref{sec:subopt} we describe a geometric alternative to gradient descent (with the same convergence rate) which gives the core of our new optimal method. Then in Section \ref{sec:why} we explain why one can expect to accelerate this geometric algorithm.

\subsection{A suboptimal algorithm} \label{sec:subopt}

Assume that we are given a guarantee $R_0 > 0$ on the distance from some point $x_0$ to the optimum, that is $x^* \in \mB(x_0, R_0^2)$. Combining this original enclosing ball for $x^*$ and the one obtained by \eqref{eq:ball2} (with $f(x^*) \leq f(x_0^+)$) one obtains
$$x^* \in \mB(x_0, R_0^2) \cap \mathrm{B}\left(x_0 - \frac{1}{\alpha} \nabla f(x_0), \frac{|\nabla f(x_0)|^2}{\alpha^2} \left(1 - \frac{1}{\kappa}\right) \right) .$$
If $\frac{|\nabla f(x_0)|^2 }{\alpha^2}\leq R_0^2(1-\frac{1}{\kappa})$ then the second ball already shrinks by a factor of $(1-\frac{1}{\kappa})$. In the other case when $\frac{|\nabla f(x_0)|^2 }{\alpha^2}> R_0^2(1-\frac{1}{\kappa})$, the center of the two balls are far apart and therefore there is a much smaller ball containing the intersection of two balls. Formally, it is an easy calculation to see that for any $g \in \R^n$, $\epsilon \in (0,1)$, there exists $x \in \R^n$ such that
$$\mB(0,1) \cap \mB(g, |g|^2 (1- \epsilon)) \subset \mB(x, 1-\epsilon) . \quad \quad \text{(Figure \ref{fig:one_ball})}$$
In particular the two above display implies that there exists $x_1 \in \R^n$ such that
$$x^* \in \mB\left( x_1 , R_0^2 \left(1 - \frac{1}{\kappa}\right) \right) .$$
Denote by $T$ the map from $x_0$ to $x_1$ defined implicitely above, and let $(x_k)$ be defined by $x_{k+1} = T(x_k)$. Then we just proved
$$|x^*-x_k|^2 \leq \left(1 - \frac{1}{\kappa}\right)^k R_0^2 .$$
In other words, after $2 \kappa \log(R_0 / \epsilon)$ iterations where each iteration cost one call to the gradient oracle) one obtains a point $\epsilon$-close to the minimizer of $f$.

\subsection{Why one can accelerate} \label{sec:why}

Assume now that we are give a guarantee $R_0 > 0$ such that $x^* \in \mB(x_0, R_0^2-\frac{2}{\alpha}(f(y)-f(x^*)))$ where $f(x_0)\leq f(y)$ (say by choosing $y=x_0$). Using the fact that $f(x_0^+)\leq f(x_0)-\frac{1}{2\beta} |\nabla f(x_0)|^2\leq f(y)-\frac{1}{2\alpha \kappa} |\nabla f(x_0)|^2\ $, we obtain that
\begin{align*}
x^* \in \mB\left(x_0, R_0^2-\frac{|\nabla f(x_0)|^2  }{\alpha^2 \kappa}- \frac{2}{\alpha}\left(f(x_0^+)-f(x^*)\right)\right)
\end{align*}
which, intuitively, allows us the shrink the radius squared from $R_0^2$ to $R_0^2-\frac{|\nabla f(x_0)|^2  }{\alpha^2 \kappa}$ using the local information at $x_0$.
From \eqref{eq:ball2}, we have
\begin{align*}
x^* \in \mB\left(x_0^{++}, \frac{|\nabla f(x_0)|^2}{\alpha^2} \left(1 - \frac{1}{\kappa}\right) - \frac{2}{\alpha} \left(f(x_0^+) - f(x^*)\right) \right) .
\end{align*}

Now, intersecting the above two shrunk balls and using Lemma~\ref{lem:geom} (see below and also see Figure~\ref{fig:two_ball}), we obtain that there is an $x_1'$ such that

 \begin{align*}
x^* \in \mB\left(x_1', R_0^2 \left(1-\frac{1}{\sqrt{\kappa}}\right)  - \frac{2}{\alpha} \left(f(x_0^+) - f(x^*)\right) \right)
\end{align*}
giving us an acceleration in shrinking of the radius. To carry the argument for the next iteration, we would have required that $f(x_1')\leq f(x_0^+)$ but it may not hold. Thus, we choose $x_1$ by a line search
$$x_1 = \LS \left(x_1', x_0^+ \right) $$
which ensures that $f(x_1)\leq f(x_0^+)$. To remedy the fact that the ball for the next iteration is centered at $x_1'$ and not $x_1$, we observe that the line search also ensures that $\nabla f(x_1)$ is perpendicular to the line going through $x_1$ and $x_1'$. This geometric fact is enough for the algorithm to work at the next iteration as well.
In the next section we describe precisely our proposed algorithm which is based on the above insights.

\section{An optimal algorithm} \label{sec:optimal}
Let $x_0 \in \R^n$, $c_0 = x_0^{++}$, and $R_0^2 = \left(1 - \frac{1}{\kappa}\right)\frac{|\nabla f(x_0)|^2}{\alpha^2}$. For any $k \geq 0$ let
$$x_{k+1} = \LS \left(c_{k}, x_k^+\right) ,$$
and $c_{k+1}$ (respectively $R^2_{k+1}$) be the center (respectively the squared radius) of the ball given by (the proof of) Lemma \ref{lem:geom} which contains
$$\mB\left(c_k, R_k^2 - \frac{|\nabla f(x_{k+1})|^2}{\alpha^2 \kappa}\right) \cap \mB\left(x_{k+1}^{++}, \frac{|\nabla f(x_{k+1})|^2}{\alpha^2} \left(1 - \frac{1}{\kappa}\right) \right).$$
The formulas for $c_{k+1}$ and $R^2_{k+1}$ are given in Algorithm \ref{algo:minball}.

\begin{theorem}\label{thm:main}
For any $k \geq 0$, one has $x^* \in \mB(c_k, R_k^2)$, $R_{k+1}^2 \leq \left(1 - \frac{1}{\sqrt{\kappa}}\right) R_k^2$, and thus
$$|x^* - c_k|^2 \leq \left(1 - \frac{1}{\sqrt{\kappa}}\right)^k R_0^2 .$$
\end{theorem}

\begin{proof} 
We will prove a stronger claim by induction that for each $k\geq 0$, one has
$$x^* \in \mB\left(c_k, R_k^2 - \frac{2}{\alpha} \left(f(x_k^+) - f(x^*)\right)\right) .$$
The case $k=0$ follows immediately by \eqref{eq:ball2}. Let us assume that the above display is true for some $k \geq 0$. Then using $f(x^*) \leq f(x_{k+1}^+) \leq f(x_{k+1}) - \frac{1}{2\beta} |\nabla f(x_{k+1})|^2 \leq f(x_k^+) - \frac{1}{2\beta} |\nabla f(x_{k+1})|^2 ,$
one gets
$$x^* \in \mB\left(c_k, R_k^2 - \frac{|\nabla f(x_{k+1})|^2}{\alpha^2 \kappa} - \frac{2}{\alpha} \left(f(x_{k+1}^+) - f(x^*)\right) \right) .$$
Furthermore by \eqref{eq:ball2} one also has
$$\mB\left(x_{k+1}^{++}, \frac{|\nabla f(x_{k+1})|^2}{\alpha^2} \left(1 - \frac{1}{\kappa}\right) - \frac{2}{\alpha} \left(f(x_{k+1}^+) - f(x^*)\right) \right).$$
Thus it only remains to observe that the squared radius of the ball given by Lemma \ref{lem:geom} which encloses the intersection of the two above balls is smaller than $\left(1 - \frac{1}{\sqrt{\kappa}}\right) R_k^2 - \frac{2}{\alpha} (f(x_{k+1}^+) - f(x^*))$.
We apply Lemma~\ref{lem:geom} after moving $c_k$ to the origin and scaling distances by $R_k$. We set $\epsilon =\frac{1}{\kappa}$, $g=\frac{|\nabla f(x_{k+1})|}{\alpha}$, $\delta=\frac{2}{\alpha}\left(f(x_{k+1}^+)-f(x^*)\right)$ and $a={x_{k+1}^{++}-c_k}$.  The line search step of the algorithm implies that $\nabla f(x_{k+1})^{\top} (x_{k+1} - c_k) = 0$ and therefore, $|a|=|x_{k+1}^{++} - c_k| \geq |\nabla f(x_{k+1})|/\alpha=g$ and Lemma~\ref{lem:geom} applies to give the result.
\end{proof}

\begin{lemma} \label{lem:geom}
Let $a \in \R^n$ and $\epsilon \in (0,1), g \in \R_+$. Assume that $|a| \geq g$. Then there exists $c \in \R^n$ such that for any $\delta >0$,
$$\mB(0,1 - \epsilon g^2 - \delta) \cap \mB(a, g^2(1-\epsilon) - \delta) \subset \mB\left(c, 1 - \sqrt{\epsilon} - \delta \right) .$$
\end{lemma}

\begin{proof}
First observe that if $g^2 \leq 1/2$ then one can take $c=a$ since $\frac{1}{2} (1- \epsilon) \leq 1 - \sqrt{\epsilon}$. Thus we assume now that $g^2 > 1/2$, and note that we can also clearly assume that $n=2$. Consider the segment joining the two points at the intersection of the two balls under consideration. We denote $c$ for the point at the intersection of this segment and $[0,a]$, and $x = |c|$ (that is $c = x \frac{a}{|a|}$). A simple picture reveals that $x$ satisfies
$$1 - \epsilon g^2 - \delta - x^2 = g^2(1-\epsilon) - \delta - (|a|-x)^2 \;\; \Leftrightarrow \;\; x = \frac{1+|a|^2 - g^2}{2 |a|} .$$

When $x \leq |a|$, neither of the balls covers more than half of the other ball and hence the intersection of the two balls is contained in the ball $\mB\left(x \frac{a}{|a|}, 1-\epsilon g^2 - \delta - x^2\right)$ (See figure \ref{fig:two_ball}). Thus it only remains to show that $x \leq |a|$ and that $1-\epsilon g^2 - \delta - x^2 \leq 1 - \sqrt{\epsilon} - \delta$. The first inequality is equivalent to $|a|^2 + g^2 \geq 1$ which follows from $|a|^2 \geq g^2 \geq 1/2$. The second inequality to prove can be written as
$$\epsilon g^2 +  \frac{(1+|a|^2 - g^2)^2}{4 |a|^2} \geq \sqrt{\epsilon} ,$$
which is straightforward to verify (recall that $|a|^2 \geq g^2 \geq 1/2$).
\end{proof}

Algorithm \ref{algo:main} we give is more agressive than Theorem \ref{thm:main}, for instance, using line search instead of fixed step size. The correctness of this version follows from a similar proof as Theorem \ref{thm:main}.

This algorithm does not require the smoothness parameter and the number of iterations; and it guarantees the function value is strictly decreasing. They are useful properties for machine learning applications because the only required parameter $\alpha$ is usually given. Furthermore, we believe that the integration of zeroth and first order information about the function makes our new method particularly well-suited in practice.

\begin{figure}
\centering
\begin{algorithm2e}[H]
\caption{Minimum Enclosing Ball of the Intersection to Two Balls}
\label{algo:minball}
\SetAlgoLined
\textbf{Input: }a ball centered at $x_{A}$ with radius $R_{A}$ and
a ball centered at $x_{B}$ with radius $R_{B}$.

\uIf{$\left|x_{A}-x_{B}\right|^{2}\geq\left|R_{A}^{2}-R_{B}^{2}\right|$}{
$c=\frac{1}{2}\left(x_{A}+x_{B}\right)-\frac{R_{A}^{2}-R_{B}^{2}}{2\left|x_{A}-x_{B}\right|^{2}}\left(x_{A}-X_{B}\right)$.
$R^{2}=R_{B}^{2}-\frac{\left(\left|x_{A}-x_{B}\right|^{2}+R_{B}^{2}-R_{A}^{2}\right)^{2}}{4\left|x_{A}-x_{B}\right|^{2}}$.
}\uElseIf{$\left|x_{A}-x_{B}\right|^{2}<R_{A}^{2}-R_{B}^{2}$}{
$c=x_{B}$.
$R=R_{B}$.
}\Else{
$c=x_{A}$.
$R=R_{A}$.\footnote{If we assume $|x_{A}-x_{B}| \geq R_B$ as in Lemma  \ref{lem:geom}, this extra case does not exist.}
}
\textbf{Output:} a ball centered at $c$ with radius $R$.
\end{algorithm2e}
\end{figure}

\begin{figure}
\centering
\begin{algorithm2e}[H]
\caption{Geometric Descent Method (GeoD)}
\label{algo:main}
\SetAlgoLined
\textbf{Input: }parameters $\alpha$ and initial points $x_{0}$.

$x_{0}^{+}=\LS(x_{0},x_{0}-\nabla f(x_{0}))$.

$c_{0}=x_{0}-\alpha^{-1}\nabla f(x_{0})$.

$R_{0}^{2}=\frac{\left|\nabla f(x_{0})\right|^{2}}{\alpha^{2}}-\frac{2}{\alpha}\left(f(x_{0})-f(x_{0}^{+})\right)$.

\For{$i\leftarrow 1, 2, \cdots$}{

\textbf{Combining Step:}

$x_{k}=\LS(x_{k-1}^{+},c_{k-1})$.

\textbf{Gradient Step:}

$x_{k}^{+}=\LS(x_{k},x_{k}-\nabla f(x_{k}))$.

\textbf{Ellipsoid Step: }

$x_{A}=x_{k}-\alpha^{-1}\nabla f(x_{k})$. $R_{A}^{2}=\frac{\left|\nabla f(x_{k})\right|^{2}}{\alpha^{2}}-\frac{2}{\alpha}\left(f(x_{k})-f(x_{k}^{+})\right)$.

$x_{B}=c_{k-1}$. $R_{B}^{2}=R_{k-1}^{2}-\frac{2}{\alpha}\left(f(x_{k-1}^{+})-f(x_{k}^{+})\right)$.

Let $B(c_{k},R_{k}^{2})$ is the minimum
enclosing ball of $B(x_{A},R_{A}^{2})\cap B(x_{B},R_{B}^{2})$.
}
\textbf{Output:} $x_{T}$.
\end{algorithm2e}
\end{figure}


\section{Experiments} \label{sec:exp}

In this section, we compare Geometric Descent method (GeoD) with a variety of full gradient methods. It includes steepest descent (SD), accelerated full gradient method (AFG), accelerated full gradient method with adaptive restart (AFGwR) and quasi-Newton with limited-memory BFGS updating (L-BFGS). For SD, we compute the gradient and perform an exact line search on the gradient direction. For AFG, we use the `Constant Step Scheme II' in  \cite{Nes04}. For AFGwR, \cite{o2013adaptive}, we use the function restart scheme and replace the gradient step by an exact line search to improve its performance. For both AFG and AFGwR, the parameter is chosen among all powers of 2 for each dataset individually. For L-BFGS, we use the software developed by Mark Schmidt with default settings (see \cite{schmidt2012minfunc}).

In all experiments, the minimization problem is of the form $\sum_i \phi (a_i^T x)$ where computing $a_i^T x$ is the computational bottleneck. Therefore, if we reuse the calculations carefully, each iteration of all mentioned methods requires only one calculation of $a_i^T x$ for some $x$. In particular, the cost of exact line searches is negligible compares with the cost of computing $a_i^T x$. Hence, we simply report the number of iterations in the following experiments.

\subsection{Binary Classification}

\begin{figure}
\centering
\begin{minipage}[t]{.49\textwidth}
\centering
  \includegraphics[trim=10 0 10 0, clip,width=\linewidth]{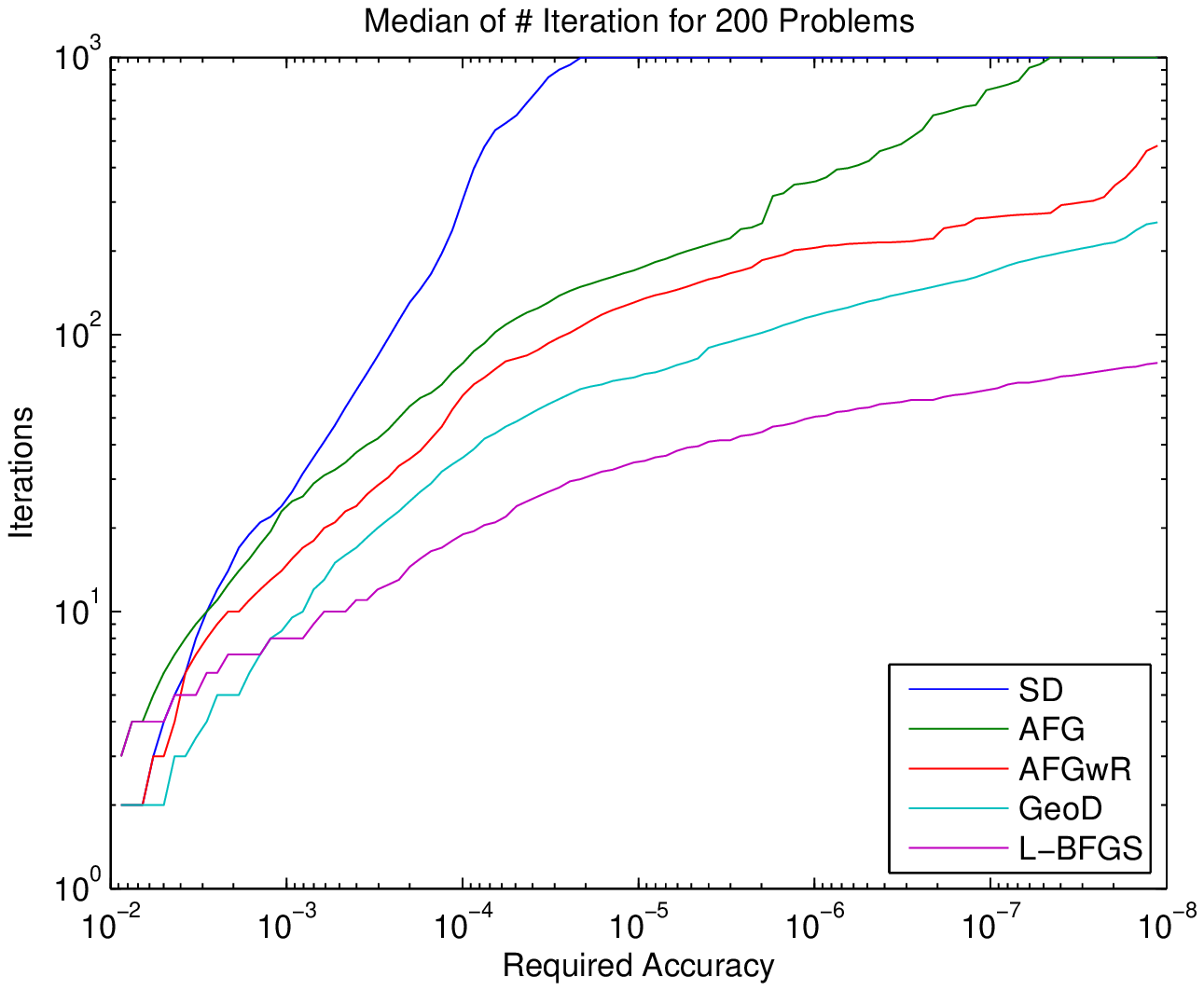}
\end{minipage}
\hfill
\noindent
\begin{minipage}[t]{.49\textwidth}
\centering
  \includegraphics[trim=10 0 10 0, clip,width=\linewidth]{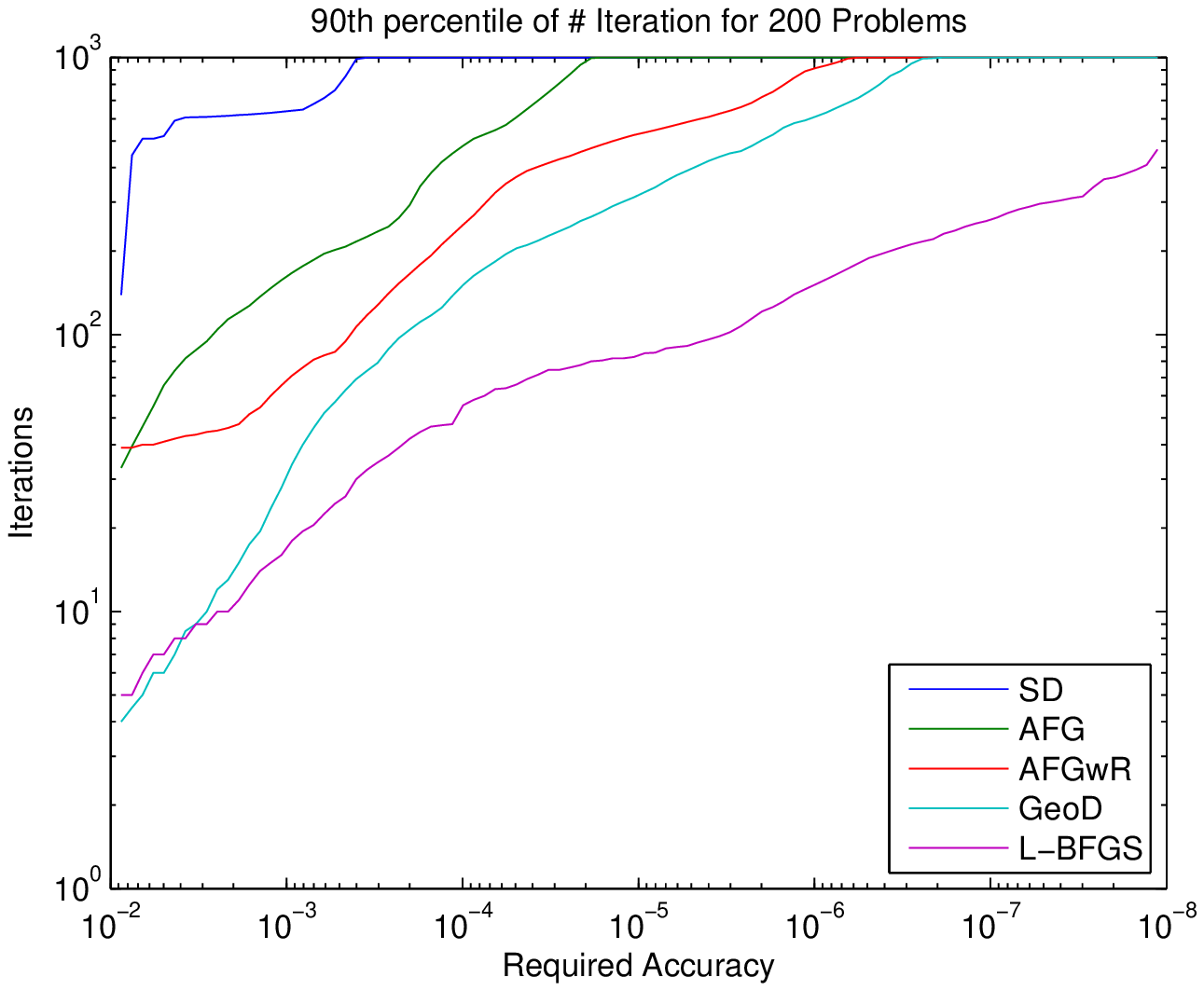}
\end{minipage}
\caption{Comparison of full gradient methods on 40 datasets and 5 regularization coefficients with smoothed hinge loss function. The left diagram shows the median of the number of iterations required to achieve a certain accuracy and the right diagram shows the 90th  percentile.}
\label{fig:binary_class}

\end{figure}

We evaluate the algorithms via the binary classification problem on the 40 datasets\footnote{We omitted all datasets of size $\geq$ 100 MB for time consideration.} from LIBSVM data, \cite{chang2011libsvm}. The problem is to minimize the regularized empirical risk:
$$f(x) = \frac{1}{n} \sum_{i=1}^n \phi(b_i \  a_i^T x) + \frac{\lambda}{2} \left| x \right|^2$$
where $a_i\in \mathbb{R}^d$, $b_i\in \mathbb{R}$ are given by the datasets, $\lambda$ is the regularization coefficient and $\phi$ is the smoothed hinge loss function given by
$$\phi(z)=\begin{cases}
0 & \text{if }z\geq1\\
\frac{1}{2}-z & \text{if }z\leq0\\
\frac{1}{2}\left(1-z\right)^{2} & \text{otherwise.}
\end{cases}
$$

We solve this problem with different regularization coefficients $\lambda \in \{10^{-4}, 10^{-5},10^{-6},10^{-7},10^{-8}\}$ and report the median and 90th percentile of the number of steps required to achieve a certain accuracy. In figure \ref{fig:binary_class}, we see that GeoD is better than SD, AFG and AFGwR, but worse than L-BFGS. Since L-BFGS stores and uses the gradients of the previous iterations, it is interesting to see if GeoD will be competitive to L-BFGS if it computes the intersection of multiple balls instead of 2 balls.

\subsection{Worst Case Experiment}
In this section, we consider the minimization problem
\begin{equation}\label{eq:worst_func}
f(x)= \frac{\beta}{2} \left((1-x_1)^2 + \sum_{i=1}^{n-1} (x_i - x_{i+1})^2 + x_{n}^2 \right) + \frac{1}{2} \sum_{i=1}^n x_i^2
\end{equation}
where $\beta$ is the smothness parameter. Within the first $n$ iterations, it is known that any iterative methods uses only the gradient information cannot minimize this function faster than the rate $1-\Theta({\beta}^{-1/2})$.

In figure \ref{fig:worst}, we see that every method except SD converge in the same rate with different constants for the first $n$ iterations. However, after $\Theta(n)$ iterations, both SD and AFG continue to converge in the rate the theory predicted while other methods converge much faster. We remark that the memory size of L-BFGS we are using is 100 and if in the right example we choose $n=100$ instead of $200$, L-BFGS will converge at $n=100$ immediately. It is surprising that the AFGwR and GeoD can achieve a comparable result by using ``memory size" being 1.

\begin{figure}[h!]
  \centering
  \includegraphics[trim=50 0 50 0, clip,width=\linewidth]{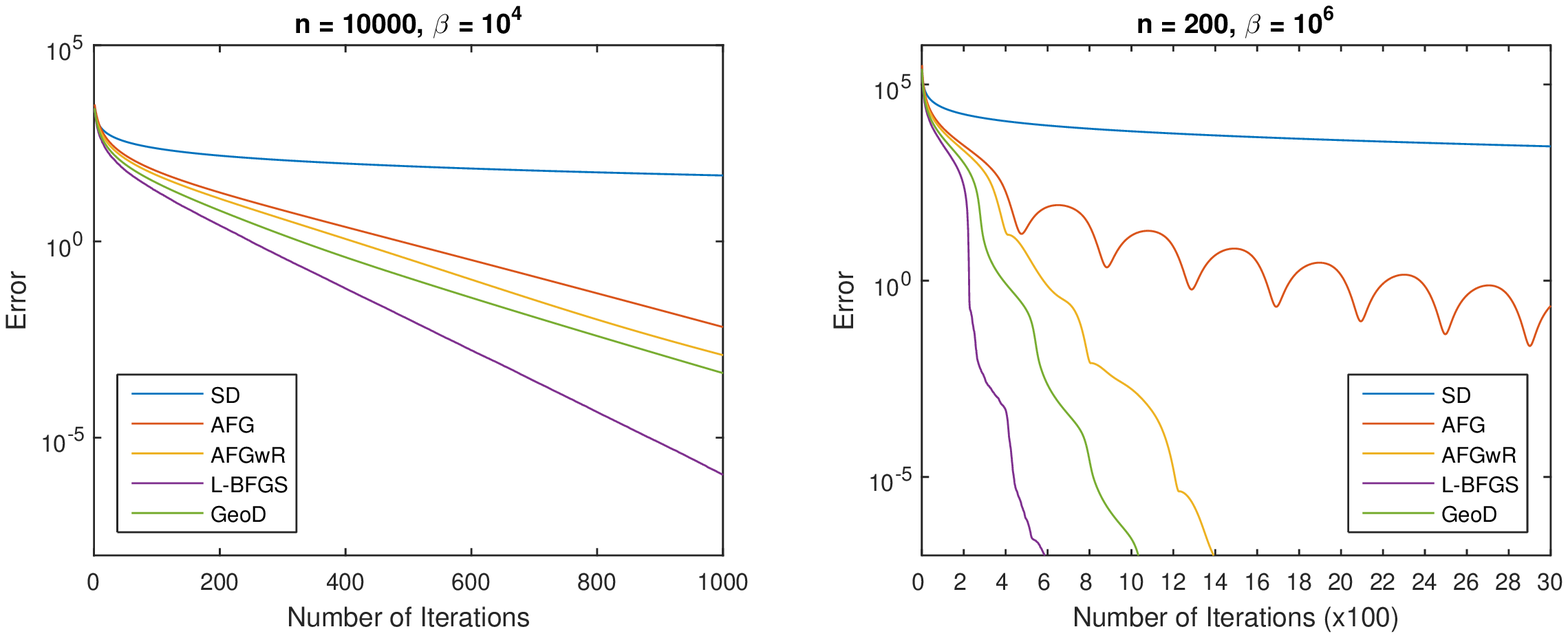}
  \caption{Comparision of full gradient methods for the function \eqref{eq:worst_func}}\label{fig:worst}
\end{figure}

\bibliographystyle{plainnat}
\bibliography{newbib}
\end{document}